\documentclass[11pt]{amsart} 

\usepackage[dvips]{graphics} 
\usepackage{color,amssymb,amsbsy,amsmath,amsfonts,amssymb,
amscd,epsfig,times,graphics}

\newtheorem{theo}{Theorem}[section] 
\newtheorem{prop}[theo]{Proposition}
\newtheorem{example}[theo]{Example}

\theoremstyle{definition}
\newtheorem{defi}[theo]{Definition}

\newcommand{\C}{\mathbb C}

\newcommand{\partx}{\partial/\partial x}

\begin{document}
  
\begin{abstract}
We provide examples of quasi-isometries for strongly convex domains in $\mathbb C^n$ endowed with their Kobayashi distance.
\end{abstract} 

\title[Quasi-isometries in strongly convex domains]
{Quasi-isometries in strongly convex domains} 

\author{Florian Bertrand and Herv\'e Gaussier}

\subjclass[2010]{32H02, 32Q45, 32Q60}
\keywords{Almost complex manifold, Kobayashi metric, Quasi-isometries}
\thanks{Research of the first author was supported by FWF grant M1461-N25.}
\maketitle 

\section*{Introduction}

Isometries between metric spaces are rigid objects that encode the underlying geometry of the metrics. For instance it can be proved that isometries between strongly convex domains in $\mathbb C^n$, endowed with their Kobayashi distance, are either holomorphic or antiholomorphic \cite{ga-se}; that structure rigidity should be satisfied by any isometry between Kobayashi hyperbolic manifolds. From a metric point of view, where it is necessary to construct flexible objects, it is more natural to deal with quasi-isometries. One can prove (see Proposition~\ref{qconf-prop}) that for $0 \leq k < 1$ every $k$-quasiconformal homeomorphism of the unit disk in $\mathbb C$ is a quasi-isometry for the Poincar\'e distance and that this result is optimal (see Example~\ref{counter-ex}). We generalize that result in higher dimension, providing examples of quasi-isometries for strongly convex domains in $\mathbb C^n$. Our main result (Theorem \ref{quasi-thm}) states that a smooth diffeomorphism between strongly convex domains, satisfying a generalized pointwise quasiconformal inequality, is a quasi-isometry for the Kobayashi metric. As examples of such maps one can quote all sufficiently small smooth deformations of biholomorphisms between strongly convex domains.

 \section{Preliminaries}

An {\it almost complex structure} $J$ on a real smooth manifold $M$ is a $\left(1,1\right)$ tensor field
 which  satisfies $J^{2}=-Id$. We suppose that $J$ is smooth.
The pair $\left(M,J\right)$ is called an {\it almost complex manifold}. We denote by $J_{st}$
the standard integrable structure on $\C^{n}$ for every $n$.
A differentiable map $f:\left(M',J'\right) \longrightarrow \left(M,J\right)$ between two almost complex manifolds is said to be 
 {\it $\left(J',J\right)$-holomorphic}  if $J\left(f\left(p\right)\right)\circ d_{p}f=d_{p}f\circ J'\left(p\right),$ 
for every $p \in M'$. In case  $M'=\Delta$ is the unit disc in $\C$, such a map is called a {\it pseudoholomorphic disc}.  

The existence of local pseudoholomorphic discs proved in \cite{ni-wo} 
enables to define the {\it Kobayashi pseudometric} $K_{\left(M,J\right)}$ for $p\in M$ and $v \in T_pM$:
$$
K_{\left(M,J\right)}\left(p,v\right):=\inf 
\left\{\frac{1}{r}>0, u: \Delta \rightarrow \left(M,J\right) 
\mbox{  J-holomorphic }, u\left(0\right)=p, d_{0}u\left(\partx\right)=rv\right\},
$$ 
and its integrated pseudodistance  $d_{\left(M,J\right)}$:

$$
d_{\left(M,J\right)}\left(p,q\right): =\inf\left\{l_K(\gamma), \mbox{ }
\gamma: [0,1]\rightarrow M, \mbox{ }\gamma\left(0\right)=p, \gamma\left(1\right)=q\right\},
$$

for $p,q \in M$, 
where  $l_K(\gamma)$ is  the {\it Kobayashi length} of a $\mathcal C^1$-piecewise smooth curve $\gamma$ defined by 
$\displaystyle l_K(\gamma):= \int_0^1 K_{\left(M,J\right)}\left(\gamma\left(t\right),\gamma'\left(t\right)\right)dt$.
The manifold $\left(M,J\right)$ is {\it Kobayashi hyperbolic} if $d_{\left(M,J\right)}$ is a distance.


\vskip 0,2cm
The main object of our study will be the canonical morphisms of  Gromov hyperbolic spaces.

\begin{defi} Let $f: (X,d) \rightarrow (X',d')$ be a map between two metric spaces. We say that 
$f$ is a quasi-isometry if there exist two positive constants  $\lambda$ and  $c$ such that for every $x,y \in X$:
\begin{equation*}
\frac{1}{\lambda}d(x,y)-c \leq d'(f(x),f(y)) \leq \lambda d(x,y)+c.
\end{equation*}
\end{defi}

\section{Quasi-isometries for strongly convex domains}

The first result concerns quasiconformal maps. This motivates the study of canonical morphisms of Gromov hyperbolic spaces and was an inspiration to study metric properties of some diffeomorphisms between strongly convex domains.
 
Let $\Omega$ be a domain in $\C$ and let $k\geq 0$. A  map $f: \Omega \rightarrow \C$ of class $\mathcal C^1$ is {\it $k$-quasiconformal}  if  for all $z \in \Omega$, $\left |\frac{\partial f}{\partial \overline{\zeta}}(z)\right |\leq 
k\left |\frac{\partial f}{\partial \zeta}(z)\right |$. Although conformal maps  of the unit disc are isometries of the Poincar\'e distance, quasiconformal maps are not necessarily 
quasi-isometries.
Indeed, the inequality
$$\frac{1}{\lambda}d_{\Delta}(\zeta,\zeta')-c \leq d_{\Delta}(f(\zeta),f(\zeta'))$$
may fail as it can be seen by considering non injective maps such as $f(\zeta)=\zeta^2$.
However we have the following proposition that may be attributed to P.Kiernan. Proposition A will be crucial in the proof our main result.
 \vskip 0,2cm
\noindent{\bf Proposition A.} {\sl Let $f: \Delta \rightarrow \Delta$ be a $k$-quasiconformal map with $k<1$. Then there is $C_k > 0$ such that:
\begin{equation}\label{eqqc}
\forall \zeta \in \Delta,\ d_{\Delta}(f(\zeta),f(\zeta'))\leq C_k (d_{\Delta}(\zeta,\zeta') + 1).
\end{equation}
}

\begin{proof}
Since $k < 1$ then according to P.Kiernan \cite{ki} there exists a constant $C_k$ such that 
\begin{equation}\label{eqqc1}
\left\{
\begin{array}{lll} 
d_{\Delta}(f(\zeta),f(\zeta'))\leq C_k d_{\Delta}(\zeta,\zeta')^{\frac{1-k}{1+k}}  & \mbox{ if }  ~ d_{\Delta}(\zeta,\zeta') \leq (\frac{1}{32})^{\frac{1+k}{1-k}}\\
\\
d_{\Delta}(f(\zeta),f(\zeta'))\leq C_k d_{\Delta}(\zeta,\zeta')  & \mbox{ if }  ~  d_{\Delta}(\zeta,\zeta') > (\frac{1}{32})^{\frac{1+k}{1-k}}.
\end{array}
\right.
\end{equation}

Then Inequality (\ref{eqqc}) is a direct consequence of (\ref{eqqc1}).
\end{proof}
As a direct application of proposition A we may consider quasiconformal homeomorphisms of the unit disc.
\begin{prop}\label{qconf-prop}
For $k<1$, $k$-quasiconformal homeomorphisms of the unit disc $\Delta$ are quasi-isometries for the Poincar\'e distance $d_{\Delta}$ on $\Delta$.
\end{prop}
\begin{proof}

Let  $f$ be such  a $k$-quasiconformal map. Then $f^{-1}$ is also a $k$-quasiconformal map and we may apply Inequality (\ref{eqqc}) to both $f$ and $f^{-1}$ to conclude the proof of Proposition~\ref{qconf-prop}.

\end{proof}

The next example  shows that the condition $k<1$ is optimal in Proposition~\ref{qconf-prop}.

\begin{example}\label{counter-ex}
The map
$$
\begin{array}{lllll}
f &: & \Delta & \rightarrow & \Delta\\
  &   & \zeta  & \mapsto     & \zeta exp\left(\frac{i}{1-|\zeta|}\right)
\end{array}
$$
satisfies the following conditions:

\vspace{0.1cm}

$(i)$ $f$ is a homeomorphism from $\Delta$ to $\Delta$,

\vskip 0,1cm
$(ii)$ $\forall \zeta \in \Delta \backslash\{0\},\
\left |\frac{\partial f}{\partial \overline{\zeta}}(\zeta)\right |<\left |\frac{\partial f}{\partial \zeta}(\zeta)\right |,$
\vskip 0,1cm
$(iii)$ $f$ is not a quasi-isometry of $(\Delta,d_\Delta)$.
\end{example}

\begin{proof}

\noindent $\bullet$ Point $(i)$ is direct.

\vskip 0,2cm
\noindent $\bullet$ Point $(ii)$. For every $\zeta \in \Delta \backslash\{0\}$ we have:
$$
\frac{\partial f}{\partial \overline{\zeta}}(\zeta) = \frac{i}{2}\frac{\zeta^2}{|\zeta|(1-|\zeta|)^2} exp\left(\frac{i}{1-|\zeta|}\right)
$$
and
$$
\frac{\partial f}{\partial \zeta}(\zeta)= \left(\frac{i}{2}\frac{\zeta \overline{\zeta}}{|\zeta|(1-|\zeta|)^2} + 1\right) exp\left(\frac{i}{1-|\zeta|}\right).
$$
This implies Point $(ii)$.

\vskip 0,1cm
\noindent $\bullet$ Point $(iii)$. 
Since 
$$
\left\{
\begin{array}{lll}
f\left(1-\frac{1}{n}\right)&=&\left(1-\frac{1}{n}\right) exp\left(in\right)\\
\\
f\left(1-\frac{1}{n+\pi}\right)&=&-\left(1-\frac{1}{n+\pi}\right) exp\left(in\right)\\
\end{array}
\right.
$$
then 
$$\lim_{n \rightarrow \infty}d_\Delta\left(f\left(1-\frac{1}{n}\right),f(1-\frac{1}{n+\pi})\right)=+\infty.$$
However, there exists $c>0$ such that:
$$
d_\Delta\left(1-\frac{1}{n},1-\frac{1}{n+\pi}\right) = \log \left(\frac{2n+2\pi-1}{2n-1}\right) \leq c.
$$
This proves Point $(iii)$.
\end{proof}

\vskip 0,3cm

The main result of this note is the following:

\begin{theo}\label{quasi-thm}
Let $D$ and $D'$ be two smooth strongly convex bounded domains in $\mathbb C^n$ and let $F$ be a smooth diffeomorphism between 
$\overline{D}$ and $\overline{D'}$.
We assume that there is a sufficiently small positive constant $c$ such that:
\begin{equation}\label{eqqcF}
\forall z \in D,\ \forall v \in \C^n,\ |\bar \partial F(z)\bar v | \leq c |\partial F(z)v |.
\end{equation}
Then $F$ is a quasi-isometry between $(D,d_{(D,J_{st})})$ and $(D',d_{(D',J_{st})})$.
\end{theo}

We recall that a bounded domain $D \subset \mathbb C^n$, with boundary $\partial D$ of class $\mathcal C^2$, is strongly convex if all the normal curvatures of $\partial D$ are positive.

\begin{proof}
Observe first that the direct image of $J_{st}$ under $F$, denoted by  $J':=F_*{J_{st}}$, is a 
 small $\mathcal{C}^2$ perturbation of $J_{st}$ on $\overline{D'}$. Indeed the complexification  $J'_{\C}$ of the structure $J'$ 
 can be written as a $(2n \times 2n)$ complex matrix:
$$J'_{\C}(z)=\left(\begin{matrix}
A (z) & B(z)\\
\overline{B}(z) & \overline{A}(z) \end{matrix}\right)
$$ 
where 
\begin{eqnarray*}
A(z)& = & i\partial F(F^{-1}(z)) \partial F^{-1}(z)-i\overline{\partial} F(F^{-1}(z)) \partial \overline{F^{-1}}(z) \\
& = & i-2i\overline{\partial}F(F^{-1}(z))\partial{\overline{F^{-1}}}(z)
 \end{eqnarray*}
 and 
$$B(z)=i\partial F(F^{-1}(z))\overline{\partial} F^{-1}(z)-i\overline{\partial}F((F^{-1}(z)))\overline{\partial F^{-1}}(z).$$ 
It then follows from (\ref{eqqcF}) that $J'$ is  $\mathcal{C}^2$ deformation of $J_{st}$ on $\overline{D'}$.  

We want to prove that there exist  two positive constants $\lambda$ and $c$ such that for all  $p,q \in D$
\begin{equation*}
\frac{1}{\lambda}d_{(D,J_{st})}(p,q)-c \leq d_{(D',J_{st})}(F(p),F(q)) \leq \lambda d_{(D,J_{st})}(p,q)+c.
\end{equation*}
Since $F$ is an isometry from $(D,d_{(D,J_{st})})$ to $(D',d_{(D',J')})$ it is equivalent to prove that $(D',d_{(D',J_{st})})$ and $(D',d_{(D',J')})$
are quasi-isometric metric spaces, namely that
 \begin{equation}\label{eqqisec4}
\frac{1}{\lambda}d_{(D',J')}(F(p),F(q))-c \leq d_{(D',J_{st})}(F(p),F(q)) \leq \lambda d_{(D',J')}(F(p),F(q))+c.
\end{equation}

Let $p,q \in D$ with $p\neq q$ and consider  the extremal holomorphic disc $f: \Delta \rightarrow D'$ 
passing through $F(p)$ and $F(q)$. According to   Theorem 5.3 in \cite{ga-jo}, for every $z \in D'$ and $v \in \mathbb C^n\backslash \{0\}$ there is a unique $J'$-stationary disc $u: \Delta \rightarrow D$ such that $u(0)=z$ and $du(0)(\partial / \partial x) =\lambda v$ for some $\lambda > 0$. Here we denote by $(x,y)$ the real coordinates in $\mathbb C$.
It follows from Theorem 6.4 in \cite{ga-jo} that $u$ is a local extremal disc, meaning that there is a neighborhood $\mathcal U$ of $u(\overline{\Delta})$ such that $u$ is extremal among all disc $\tilde u: \Delta \rightarrow \mathcal U$ such that $\tilde u(0) = z$ and $d\tilde u(0)(\partial / \partial x) \in \mathbb R^+v$. Hence the disc  $F^{-1}\circ u$ is locally extremal in $D$. It follows from Lempert's theory (\cite{lem}) that $F^{-1}\circ u$ is extremal. Therefore $u$ is an extremal disc. Since $D$ is foliated by holomorphic extremal discs centered at $F^{-1}(z)$, the foliation being singular at $F^{-1}(z)$, we obtain a singular foliation of $D'$ by the $J'$-holomorphic extremal discs constructed in \cite{ga-jo}.
Consequently consider the unique $J'$-holomorphic extremal disc $u$ passing through 
$F(p)$ and $F(q)$.
Since $J'$ is a small $\mathcal{C}^2$ perturbation of 
$J_{st}$, $u$ is a small $\mathcal{C}^1$ deformation of $f$ due to the proof of Theorem 5.3 in \cite{ga-jo} which is based on 
the Implicit Function Theorem.
 According to  \cite{lem2} there exists a holomorphic retract $r_f: D' \rightarrow \Delta$ such that $r_f \circ f = id$.  Let 
 $\zeta,\zeta',\eta,\eta' \in \Delta$ be 
such that $F(p)=u(\zeta)=f(\eta)$ and $F(q)=u(\zeta')=f(\eta')$. 

Although the composition of a holomorphic function and a $J'$-holomorphic disc is not, in general, quasiconformal, the 
map $r_f \circ u$ is  $k$-quasiconformal with $k< 1$; indeed, it has a small $\frac{\partial}{\partial \overline{\zeta}}$  
derivative 
and, since $u$ is a small $\mathcal{C}^1$ deformation of $f$,  $\left|\frac{\partial r_f \circ u}{\partial \zeta}\right|$ is close to 
$1$. 
It follows from  Proposition A that
$$d_{\Delta}(r_f \circ f(\eta),r_f \circ f(\eta'))=d_{\Delta}(r_f \circ u(\zeta),r_f \circ u(\zeta'))\leq \lambda d_{\Delta}(\zeta,\zeta') + 
c,$$
which, by the extremal properties of $f$ and $u$ gives the right hand side of (\ref{eqqisec4}). 

Moreover the disc $F^{-1}\circ u$ is holomorphic (in the standard sense) and extremal for the pair of points $(p,q)$ in $D$. Denote by 
$r: D\rightarrow \Delta$ its holomorphic retract and set $r_u:=r\circ F^{-1}$. The function $r_u$ satisfies $r_u \circ u = id$ and 
 $r_u \circ f$ is  also $k$-quasiconformal with $k<1$. This provides the left hand side of (\ref{eqqisec4}) and ends the proof of Theorem~\ref{quasi-thm}.
\end{proof}

\vskip 0,2cm

We end this note with a remark concerning morphisms of Gromov hyperbolic spaces in relation with holomorphic maps. It is natural to try to enlarge the class of mappings studied in this Section.
In particular we could consider diffeomorphisms, between strongly pseudoconvex domains, satisfying inequalities such as $|\bar \partial _J f| \leq C |df|$ where $J$ is the almost complex ambient structure.
New problems arise in that situation that will lead to a more specific study. For instance, as mentioned by J.-P.Rosay in \cite{ros}, it is not known if such a map has isolated zeroes if it is not identically zero, in complex dimension more than one.

\vskip 0,5cm
{\small
\noindent Florian Bertrand\\
Department of Mathematics, University of Vienna\\
Nordbergstrasse 15, Vienna, 1090, Austria\\
{\sl E-mail address}: florian.bertrand@univie.ac.at\\
\\
Herv\'e Gaussier\\
(1) UJF-Grenoble 1, Institut Fourier, Grenoble, F-38402, France\\
(2) CNRS UMR5582, Institut Fourier, Grenoble, F-38041, France\\
{\sl E-mail address}: herve.gaussier@ujf-grenoble.fr\\
} 
 
\end{document}